\documentclass[12pt]{amsart}
\usepackage[mathscr]{eucal}
\usepackage{amssymb}
\usepackage{latexsym}
\usepackage{amsthm}
\usepackage[pagewise]{lineno}

\theoremstyle{plain}

\newtheorem*{theo}{Theorem}

\newtheorem{lem}{Lemma}[section]
\newtheorem*{cor}{Corollary}

\numberwithin{equation}{section}

\title[Another proof of the corona theorem]{Another proof of the corona theorem}
\author{Jun-ichi Tanaka}

\medskip

\dedicatory{Dedicated to the memory of Junzo and Sadako Wada}

\address{Department of Mathematics, School of Education, Waseda
University, Shinjuku, Tokyo 169-8050, Japan}

\email{jtanaka@waseda.jp}

\date{Finally revised on June 21, 2024}



\begin{document}
\maketitle

\begin{abstract}
Let $H^\infty(\Delta)$ be the uniform algebra of
bounded analytic functions on the open unit disc
$\Delta$, and let $\mathfrak{M}(H^\infty)$
be the maximal ideal space of $H^\infty(\Delta)$.
By regarding $\Delta$ as an open subset of $\mathfrak{M}(H^\infty)$,
the corona problem asks whether $\Delta$ is dense in
$\mathfrak{M}(H^\infty)$, which was solved
affirmatively by L. Carleson. Extending the cluster value theorem
to the case of finitely many functions, we provide a direct proof
of the corona theorem:
Let $\phi$ be a homomorphism in
$\mathfrak{M}(H^\infty)$, and
let $f_1, f_2, \dots, f_N$ be functions in $H^\infty(\Delta)$. Then there is a
sequence $\{\zeta_j\}$ in $\Delta$ satisfying
$f_k(\zeta_j) \rightarrow \phi (f_k)$ for $k=1, 2, \dots, N$. On the other hand,
the corona problem remains unsolved in many general settings, for instance, certain plane domains,
polydiscs and balls, our approach is so natural that it may be possible to deal with such cases
from another point of view.

\end{abstract}

\bigskip
\bigskip
\section{Introduction}
The corona problem was posed by S. Kakutani in 1941 and
finally settled in 1962 by L. Carleson \cite{C}, where he introduced
important techniques to solve the problem.
Many new methods have been exploited since then, especially,
T. Wolff \cite{G4} presented a new proof of the corona theorem in 1979.
However, the author learned from O. Hatori that Kakutani had often
said there would be a simple proof of the corona problem.
Indeed, E. L. Stout also wrote in \cite[p 32 ]{S}: \textit{Carleson's proof uses
only classical analysis. It would be of great interest to have
a solution to the corona problem that
draws less on classical methods and more on algebraic analysis,
but to the best of my knowledge, no such proof has been discovered yet.}
Our approach may make headway to some degree in this direction. Roughly
speaking, the corona problem could be solvable only with the knowledge
of Hoffman's book \cite[Chapter 10]{H}.

The usual Lebesgue and Hardy spaces in the unit circle $\mathbf{T}$ are
denoted by $L^p(\mathbf{T})$ and $H^p(\mathbf{T}), 1\le p\le \infty$, respectively.
We usually identify $\mathbf{T}$ with $[0,2\pi)$ and, for a function
$f$ on $\mathbf{T}$, we write $f(\theta)$ for $f(e^{i\theta})$.
By boundary value identification, $H^\infty(\Delta)$ may be considered as
the closed subalgebra $H^\infty(\mathbf{T})$ of $L^\infty(\mathbf{T})$.
Regarding $H^\infty(\Delta)$ as a uniform
algebra on $\mathfrak{M}(H^\infty)$, we observe that its Shilov boundary $X$
is the maximal ideal space $\mathfrak{M}(L^\infty)$ of
$L^\infty(\mathbf{T})$, which is totally disconnected.
Denote by $m$ the normalized Lebesgue measure
$dm(\theta) = d\theta /2\pi$ on $\mathbf{T}$. Since
$L^\infty(\mathbf{T})$ is identified with $C(X)$, $m$ is lifted to
a measure $\widehat{m}$ on $X$, where each measurable set
$E$ in $\mathbf{T}$ corresponds to an open-closed
subset $U(E)$ of $X$ . Then the family $\{U(E)\}$ of all such
open-closed subsets forms a basis for the topology of $X$, and
satisfies that $m(E)=\widehat{m}(U(E))$ and $\widehat{m}(U(E))>0$
unless $U(E) = \emptyset$
(see \cite[Chapter I]{G1}). Recall that the Poisson kernel
for $z$ in $\Delta$ is given by
$P_z(\theta)\,=\,Re [\,(e^{i\theta}+z)/(e^{i\theta}-z)]$,
so the Poisson integral by $P_z(\theta)dm(\theta)$ is also
regarded as a measure on $X$.
From now on, we identify
each function in $H^\infty(\Delta)$ with its Gelfand transform,
and regard $H^\infty(\Delta)$ as a uniformly closed
subalgebra of $C(\mathfrak{M}(H^\infty))$.

When $\alpha$ is in ${\mathbf{T}}$, the {\it fiber\/} $\mathfrak{M}_{\alpha}$ of $\mathfrak{M}(H^\infty)$  over $\alpha$ is defined to be
$$
\mathfrak{M}_{\alpha}\;=\;\left\{\xi\in \mathfrak{M}(H^\infty) \;; \; \xi(z)\,=\,\alpha \right\},
$$
where $z$ is the coordinate function. For each function $f$ in $H^\infty(\Delta)$,
the \textit{cluster set} of $f$ at $\alpha$ is
$$
{\it Cl}(f,\alpha)\;=\;\bigcap_{r>0}\,\overline{f(\Delta\cap \{\,\vert z-\alpha\vert < r\})}.
$$
Then the \textit{cluster value theorem} asserts that
\begin{equation}
\textit{Cl}\,(f,\alpha)\;=\;f(\mathfrak{M}_{\alpha}), \quad f \in H^\infty(\Delta),
\label{1.1}
\end{equation}
consequently, if $\phi$ is in $\mathfrak{M}_{\alpha}$, then
there is a sequence $\{\zeta_j\}$ in $\Delta$ satisfying
$\zeta_j \rightarrow \alpha$ and $f(\zeta_j) \rightarrow f(\phi)$.
Notice that this sequence $\{\zeta_j\}$ depends on $f$ and $\phi$,
however, the same property holds on the uniformly closed subalgebra of $H^\infty(\Delta)$ generated by $f$.
Recall that the open unit disc $\Delta$ is homeomorphically embedded in
$\mathfrak{M}(H^\infty)$ by identifying each $z$ in $\Delta$
with the point evaluation $\phi_z(f)\,=\,f(z)$
(see \cite[Chapter 10]{H}).
We then have the decomposition
$$
\mathfrak{M}(H^\infty)\setminus \Delta\;=\;\bigcup_{\vert\alpha\vert = 1}\,
\mathfrak{M}_{\alpha} .
$$
Since $\mathfrak{M}_{\alpha}$ is a peak set
with peaking function $(1+\bar{\alpha}z)/2$, the restriction of $H^\infty(\Delta)$ to
$\mathfrak{M}_{\alpha}$ is a uniform algebra on $\mathfrak{M}_{\alpha}$, which is denoted by $A_\alpha$. Then the Shilov boundary of $A_\alpha$ is $X\cap \mathfrak{M}_{\alpha}$ (see \cite[187p - 193p]{H} for
the algebras $A_\alpha$). Each $\phi$ in $\mathfrak{M}_\alpha$
has a unique representing measure $\mu$ on $X\cap \mathfrak{M}_{\alpha}$ with minimal support $S_{\phi}$ (see \cite[Chapter II, Theorem 2.3]{G1} and
 \cite[Chapter 10, Exercise 4]{H} for minimal support sets).
Since various fibers are homeomorphic to one another,
we restrict our attention to the fiber
$\mathfrak{M}_1$ over $1$ to look into the structure of fringe
$\mathfrak{M}(H^\infty)\setminus \Delta$.

\medskip
Our objective in this note is to provide a strong version of the cluster value theorem (\ref{1.1}), from which the corona theorem follows directly:
\begin{theo}
\label{theo}
Let $B(\mathfrak{F})$ be the uniformly closed subalgebra of $H^\infty(\Delta)$ generated by a countable family $\mathfrak{F}$
of $H^\infty(\Delta)$. If $\phi$ is a homomorphism in the fiber $\mathfrak{M}_1$ over $z=1$, then there is a sequence
 $\{\zeta_j\}$ in $\Delta$, depending on $\mathfrak{F}$ and $\phi$, such that
\begin{equation}
\label{eq1.2}
\zeta_j \rightarrow 1 \qquad
and \qquad f(\zeta_j) \rightarrow f(\phi)
\end{equation}
for each $f$ in
$B(\mathfrak{F})$.
\end{theo}

Let us make a comment on Theorem.
It is not necessary that the homomorphism $\phi$
lies in the closure of $\{\zeta_j\}$ in $\mathfrak{M}(H^\infty)$,
in other words, there may exist $h$ in $H^\infty(\Delta)$ with the
property that $h(\phi) = 1$ while $\vert h(\zeta_j)\vert < 1/2$ for
 $j=1,2,\cdots.$ Of course, each $f$ in $B(\mathfrak{F})$ values constant $f(\phi)$ on the set of adherent points
of $\{\zeta_j\}$ in $\mathfrak{M}(H^\infty)$. Since $\{\zeta_j\}$ may be
chosen to be sufficiently sparse, there appears a relation between
interpolating sequences and analytic discs (see Section 2 for details).

Recall that a basic neighborhood of $\phi$ in
$\mathfrak{M}(H^\infty)$ is given by
\begin{equation}
\label{eq1.3}
W(\phi, f_1,\cdots, f_N, \varepsilon)\,=\,\left\{ \xi \in
\mathfrak{M}(H^\infty)\,;\, \vert f_k(\xi)- f_k(\phi)\vert <\varepsilon, \;
k=1,2,\cdots, N \,\right\}\,,
\end{equation}
for $\varepsilon > 0$ and for $f_1, f_2, \cdots, f_N$  in
$H^\infty(\Delta)$. By definition, the family of all such neighborhoods
forms a basis for the (weak-star) topology of $\mathfrak{M}(H^\infty)$.
Since $W(\phi, f_1, \cdots, f_N, \varepsilon)
\cap \Delta \neq \emptyset$ by Theorem, it follows immediately
that the open set $\Delta$ is dense in $\mathfrak{M}(H^\infty)$. This fact is interpreted as a well-known formulation in function theory:

\begin{cor}
If $f_1, f_2, \cdots, f_N $ in $H^\infty(\Delta)$ satisfy
$$
\vert f_1(z) \vert +\vert f_2(z) \vert + \cdots +\vert f_N(z) \vert
\ge \delta > 0\,,\quad z \in \Delta\,,
$$
then there exist $g_1, g_2, \cdots, g_N $ in $H^\infty(\Delta)$
such that
$$
f_1(z)g_1(z) +f_2(z)g_2(z)  + \cdots +f_N(z)g_N(z) \,\equiv \,1\,,\quad z \in \Delta\,.
$$
\end{cor}

It would be helpful to understand the basic idea behind
our proof of the corona theorem. Let $\phi$ be a homomorphism in
the fiber $\mathfrak{M}_1$, and let $\mu$ be the representing
measure for $\phi$. Then the minimal support $S_\phi$ for $\mu$
is contained in $X \cap \mathfrak{M}_1$.
Since $f_1,f_2, \cdots, f_N $ in $H^\infty(\Delta)$ are continuous
on $X$, we may choose disjoint open-closed subsets
$U_i=U(E_i), i = 1, 2, \dots ,\ell,$ of $X$ such that $S_{\phi} = \cup_{i=1}^{\,\ell} (U_i \cap S_{\phi})$
and all $f_k$ vary little on each $U_i$. Here $E_i$ denotes the
measurable set in $\mathbf{T}$ corresponding to $U_i$.
Denoting by $\chi_{E_i}$  the characteristic function of $E_i$.
we then choose a nonnegative simple function of the form
$s(\theta) \,=\, \sum_{i=1}^{\,\ell} a_i \chi_{E_i}(\theta)$
satisfies that $\int_{\mathbf{T}} s(\theta) dm(\theta) = 1$
and the value of
\begin{equation*}
\left\vert \int_{S_{\phi}} f_k(x) d\mu(x) - \int_{\mathbf{T}} \,
f_k(\theta) s(\theta) dm(\theta) \right\vert
\end{equation*}
is as small as desired. Therefore there is
a sequence $\{s_j(\theta)\}$ of such simple functions satisfying
\begin{equation*}
\lim_{j\to \infty} \left\vert \phi(f_k) -
\int_{\mathbf{T}} \, f_k(\theta) s_j(\theta)\,
dm(\theta)  \right\vert \;=\; 0\,.
\end{equation*}
With the aid of certain Blaschke products, we then see that
$s_j(\theta) dm(\theta)$ is close to the Poisson
integral of $\zeta_j$ asymptotically on the algebra
generated by $f_1,f_2, \cdots, f_N $.
This shows that the sequence $\{\zeta_j\}$ in $\Delta$
satisfies
\begin{equation*}
\lim_{j\to \infty} \left\vert f_k(\zeta_j) -
\int_{\mathbf{T}} \, f_k(\theta) s_j(\theta)
dm(\theta) \right\vert \;=\; 0\,,
\end{equation*}
from which the conclusion of Theorem follows (see
Lemma \ref{lem2.3} and Section 5 for more details).

\medskip
In the next section, we establish some notation and elementary
facts on the structure of $\mathfrak{M}(H^\infty)$. In Section 3,
among other things, Hoffman maps are discussed by the relation to
interpolating sequences in $\Delta$.
Section 4 is devoted to constructing auxiliary Blaschke products
of which zeros determine desired sequences.
In Section 5, the proof of Theorem is provided.
We close with two remarks in Section 6.

\medskip
We refer the reader to \cite{A}, \cite{C}, \cite{D}, \cite[Chapter VIII]{Ga}
and \cite[Appendix 3]{N} for further details and recent developments
on the corona problem.
Related results concerning the Hardy space theory can be found in
\cite{G1}, \cite{Ga}, \cite{H} and \cite{N}.

5

\bigskip
\bigskip
\section{Analytic discs and Hoffman maps}
\label{S2}
We begin with showing that the Shilov boundary $X$ of
$H^\infty(\Delta)$ is contained in the
closure of $\Delta$ in $\mathfrak{M}(H^\infty)$, which is well-known. This
fact enables us to restrict our attention to the homomorphisms lying in
$\mathfrak{M}(H^\infty)\setminus X$.

\begin{lem}
\label{lem2.1}
Let $\phi$ be a homomorphism in $\mathfrak{M}(H^\infty)$,
and let $f_1, f_2, \cdots, f_N $ be functions in $H^\infty(\Delta)$.
Denote by $S_\phi$ the minimal support of representing measure $\mu$
for $\phi$. If $f_1, f_2, \cdots, f_N $ are constant on $S_\phi$,
then we have
$$
W(\phi, f_1, \cdots, f_N, \varepsilon)
\cap \Delta \neq \emptyset,
$$
for any \,$\varepsilon >0 $.
Consequently, the Shilov boundary X lies in the
closure of $\Delta$ in $\mathfrak{M}(H^\infty)$.
\end{lem}

\begin{proof}
Since each $f_k$ is continuous on $X$, $f_k(\phi) = f_k(x)$
for all $x$ in $S_\phi$. Fix an $x$ in $S_\phi$, and
choose an open-closed
neighborhood $U=U(E)$ of $S_\phi$ such that
$$
\vert f_k(\psi) - f_k(x)  \vert < \varepsilon/2\,, \quad \psi \in U,
$$
for $k=1, 2, \cdots, N$.
Since the subset $E$ of $\mathbf{T}$ satisfies that
$m(E) = \widehat{m}(U(E)) > 0$, we observe that
$$
\vert f_k(\theta) - f_k(x) \vert < \epsilon/2\,, \quad m-a.e. \quad
\theta \in E.
$$
Thus it follows from Fatou's theorem that, for some $\theta$ in $E$,
there is a $z=re^{i\theta}$ in $\Delta$ satisfying that
\begin{align*}
\vert f_k(z) - f_k(\phi) \vert
& \le \vert f_k(z) - f_k(\theta) \vert +\vert f_k(\theta) - f_k(x) \vert\\
& < \varepsilon/2 + \varepsilon/2 = \varepsilon\,,
\end{align*}
for $k=1, 2, \cdots, N$, so the proof is complete.
\end{proof}

Let us make a remark on this lemma. Since $H^\infty(\Delta)$ is a
logmodular algebra on $X$, $\mu$ is a Jensen measure , meaning that
$\mu$ satisfies the inequality
$$
\log \vert \phi(f)\vert \;\le\; \int_{S_{\phi}}
 \log \vert f(x) \vert\, d\mu(x) \,,\quad f \in H^\infty(\Delta) ,
$$
so if $f$ vanishes on a Borel subset $K$ with $\mu(K)>0$,
then $f(\phi) = 0$. This provides that if each $f_k$ is constant
$c_k$ on such a $K$, then the conclusion of Lemma \ref{lem2.1} holds.
We notice that, except for analytic discs, there may exist a
function $f$ in $H^\infty(\Delta)$ such that $f$ is not constant
on $S_\phi$ and the right side of the above inequality diverges.

\medskip
For $\eta$ and $\xi$ in $\mathfrak{M}(H^\infty)$, the
\textit{pseudo-hyperbolic distance} $\rho(\eta,\xi)$ between
$\eta$ and $\xi$ is defined to be
$$
\rho(\eta,\xi)\;=\; \sup\left\{\,\vert f(\eta)\vert \,;\;
f \in H^\infty(\Delta), \,f(\xi)=0 \; and \; \Vert f \Vert \,\leq\, 1 \;\right\}.
$$
Then the relation $\rho(\eta,\xi)< 1$ is an equivalence relation in $\mathfrak{M}(H^\infty)$ and
the equivalence class $P(\xi)=\{ \eta \in \mathfrak{M}(H^\infty);\rho(\eta,\xi)< 1\}$ is called the \textit{Gleason part} of $\xi$. A Gleason part $P$ is an \textit{analytic disc} if there exists a continuous
bijective map $L$ of $\Delta$ onto $P$ such that
$f\circ L$ is analytic on $\Delta$ for all $f$ in
$H^\infty(\Delta)$, and such a map $L$ is called an \textit{analytic map}.
Since $H^\infty(\Delta)$ is a logmodular algebra on $X$, it follows
from Wermer's embedding theorem that
each part is either a single point or an analytic disc.

Furthermore, K. Hoffman \cite{H2} characterized analytic discs in $\mathfrak{M}(H^\infty)$ by using interpolating sequences in $\Delta$.
Recall that a sequence $\{z_j\}$ in $\Delta$ is an
\textit{interpolating sequence} if,
for any bounded sequence $\{w_j\}$, there exists
a function $f$ in $H^\infty(\Delta)$ such that
$f(z_j)=w_j$, for $j = 1, 2, \cdots.$
Such a sequence is characterized by the condition
$$
\inf_{k} \prod_{j:j\not= k} \left \vert
\frac{z_k - z_j}{1 - \overline{z_j}z_k} \right \vert
> 0\;.
$$
Especially, an interpolating sequence $\{z_j\}$ is
said to be \textit{thin} (\textit{sparse}), if it satisfies
$$
\lim_{k\to \infty} \prod_{j:j \not= k} \left \vert
\frac{z_k - z_j}{1 - \overline{z_j}z_k} \right \vert
\;=\;1\,.
$$
For an  an interpolating sequence $\{z_j\}$, the associated
Blaschke product
\begin{equation}
\label{eq2.1}
B(z)\;=\; \prod_{j= 1}^\infty \frac{\overline{z_j}}{\vert z_j\vert} \, \frac{z_j - z}{1 - \overline{z_j}z}\;,
\end{equation}
is called the \textit{interpolating Blaschke product\/} (where $\overline{z_j}/\vert z_j\vert=-1,$ if $z_j=0$\,). When $B(z)$
is a Blaschke product, let us agree to also call $e^{i\gamma}B(z)$
a Blaschke product, for a real constant $\gamma$.

\medskip
The set $\mathfrak{M}(H^\infty)^{\Delta}$ of all maps of $\Delta$ into
$\mathfrak{M}(H^\infty)$ is a compact Hausdorff space in the product
topology. Observe that, in this topology, a net $(F_\beta)$ in
$\mathfrak{M}(H^\infty)^{\Delta}$ has limit $F$ if and only if
$f\circ F_\beta(\zeta)\rightarrow f\circ F(\zeta)$ for all $f$
in $H^\infty(\Delta)$ and all $\zeta$ in $\Delta$.
For a sequence $\{c_n\}$ in $\Delta$, we put
\begin{equation}
\label{eq2.2}
L_n(\zeta)\;=\;
\frac{\zeta+c_n}{1+\overline{c_n}\zeta}\,,\quad
\zeta \in \Delta\,.
\end{equation}
Then $L_n$ is an analytic map of $\Delta$ onto the part $\Delta$ in
$\mathfrak{M}(H^\infty)$. From the sequence $\{L_n\}$ in $\mathfrak{M}(H^\infty)^{\Delta}$, we take a convergent subnet $(L_\beta)$
with limit $L$ in $\mathfrak{M}(H^\infty)^{\Delta}$, which is called the
\textit{Hoffman map} determined by $(L_\beta)$.

Let $P(\psi)$ be a Gleason part of $\psi$ in $\mathfrak{M}(H^\infty)$.
Then Hoffman showed that $P(\psi)$ is an analytic disc if and
only if the analytic map for $P(\psi)$ is the Hoffman map
$L_{\psi}=\lim_\beta L_\beta$, where $(L_\beta)$ is
a subnet of $\{L_n\}$ for an interpolating sequence $\{c_n\}$.
We notice that the proof of the ``\,only if\,'' part requires the
corona theorem. Whenever $P$ is an analytic disc in $\mathfrak{M}(H^\infty)\setminus \Delta$, the closure of
$P$ in $\mathfrak{M}(H^\infty)$ never meets the
Shilov boundary $X$, because of the existence of a Blaschke
product vanishing identically on $P$ (see \cite[102p]{H2}).

\medskip
The following result is well-known, and used repeatedly
in what follows (see \cite[Chapter X, Excercise 8]{Ga}).
However, let us provide here an easy proof. It is noteworthy
that the argument does not depend on the corona theorem.
\begin{lem}
\label{lem2.2}
Let $\/ \{c_n\}$ be a thin interpolating sequence
in $\Delta$ with $c_n\to 1$, and let $L$ be the Hoffman
map by a convergent subnet $(L_\beta)$ of the sequence
$\{L_n\}$ by \eqref{eq2.2}. We put $\xi=L(0)$, a homomorphism
in $\mathfrak{M}_1$. Then the Gleason part $P(\xi)$ of $\xi$
is an analytic disc in $\mathfrak{M}_1$, where $L$ is an
analytic homeomorphism of $\Delta$ onto $P(\xi)$.
\end{lem}

\begin{proof}
Since $\mathfrak{M}_1$ is a peak set, $P(\xi)$ is contained
in $\mathfrak{M}_1$. It follows from \cite[Chapter X, Lemma 1.1]{Ga}
that $L$ is an analytic map of $\Delta$ to $P(\xi)$. Now we show
that $L$ maps $\Delta$ {\it onto} $P(\xi)$. Let $B$ be the Blaschke
product with zeros $\{c_n\}$. Observe that $B\circ L_\beta(\zeta)$
converges uniformly to $B\circ L(\zeta)$ on compact subsets of
$\Delta$. Since
$B\circ L(0)=0$, and since
\begin{equation*}
\vert(B\circ L)'(0)\vert \;=\; \lim_\beta\,
\vert(B\circ L_\beta)'(0)\vert
 \;=\; \lim_\beta \,(1-\vert c_\beta\vert^2)\,
 \vert B'(c_\beta)\vert \,=\,1,
\end{equation*}
Schwarz's lemma shows that $(B\circ L)(z) = z$. On the other
hand, Wermer's embedding theorem assures the existence of
an analytic map $\tau$ of $\Delta$ onto $P(\xi)$
with $\tau(0)=\xi$. Let $f=\tau^{-1}\circ L$ and
$g=B\circ \tau$. Then these functions map $\Delta$
into itself, and vanish at $0$ in $\Delta$. Since
$H^\infty(\Delta)$ is a logmodular algebra on $X$,
$\tau^{-1}$ is approximated by a bounded sequence in
$H^\infty(\Delta)$ (compare to (a) of Section 6). Observe that
$g(f(z)) = (B\circ L)(z) = z$. Since $\vert g'(f(z))f'(z)\vert = 1$,
$\vert f'(0)\vert = 1$ by Schwarz's lemma. This shows that $L$
maps $\Delta$ onto $P(\xi)$. Since $B=L^{-1}$ is continuous on
$P(\xi)$, $L$ is a desired homeomorphism.
\end{proof}

In our argument it would be useful to understand the following
observation: Let $f_1, f_2, \cdots, f_N$ be in $H^\infty(\Delta)$,
and let $\{c_n\}$ be a sequence in $\Delta$.
By taking a suitable subsequence $\{c_{n_i}\}$ of $\{c_n\}$, it
follows from normal families argument that $f_k\circ L_{n_i}$
converges uniformly to $F_k$ on compact subsets of $\Delta$,
for $k=1, 2, \cdots, N$. We also assume that $\{c_{n_i}\}$ is
a thin interpolating sequence. If $L$ is the Hoffman map by a
convergent subnet of $\{L_{n_i}\}$, then $L$ is a homeomorphism with
$F_k = f_k\circ L$. Moreover, $f_k$ extends to the closure of
Gleason part of $L^{-1}(0)$ in $\mathfrak{M}(H^\infty)$.

\begin{lem}
\label{lem2.3}
Let $\phi$ be a homomorphism in $\mathfrak{M}_1$.
Then $\phi$ lies in the closure of \,$\Delta$ in
$\mathfrak{M}(H^\infty)$ if and only if, for every
at most countable family $\mathfrak{F}$ in $H^\infty(\Delta)$,
there is a thin interpolating sequence $\{\zeta_j\}$ such that
\eqref{eq1.2}\, holds for each $f$ in $B(\mathfrak{F})$,
the uniformly closed subalgebra generated by $\mathfrak{F}$.
Moreover, there is a $\xi$ in $\mathfrak{M}_1$
whose Gleason part $P(\xi)$ is
homeomorphic to $\Delta$ such that $f(\phi)\,=\,f(\xi)$
for each $f$ in $B(\mathfrak{F})$.
\end{lem}

\begin{proof}
Suppose that $\phi$ lies in the closure of $\Delta$ in
$\mathfrak{M}(H^\infty)$. Put $\,\mathfrak{F}=\{f_1, f_2, \cdots \}$,
and let $\mathfrak{F}_0$ be the family of finite sums of functions
of the form $ r\cdot f_1^{n_1} f_2^{n_2} \cdots f_k^{n_k}$ with a
rational number $r$ and with nonnegative integers $n_1, n_2, \cdots, n_k$.
Observe that $\mathfrak{F}_0$ is also a countable family. Replacing
$\mathfrak{F}$ with $\mathfrak{F}_0$, we may assume $B(\mathfrak{F})$
is the uniform closure of $\mathfrak{F}$ in $H^\infty(\Delta)$.
Let $\{\varepsilon_j\}$ be a decreasing sequence of positive numbers
with $\varepsilon_j\to 0$. Since $\mathfrak{M}_1$ is a peak set in
$\mathfrak{M}(H^\infty)$, it then follows from our assumption that
$$
W(\phi, f_1, \cdots, f_j, \varepsilon_j)\;
\cap \;\left\{ z ; \vert z - 1 \vert < \varepsilon_j \right\}\; \neq \emptyset \,,
$$
where $W(\phi, f_1, \cdots, f_j, \varepsilon_j)$ is
defined as in \eqref{eq1.3}. We then fix a $\zeta_j$ in
this set. Taking a subsequence if necessary, we choose a
sequence $\{\zeta_j\}$ for which \eqref{eq1.2} holds on
$\mathfrak{F}$. Since $\mathfrak{F}$ is uniformly dense
in $B(\mathfrak{F})$, the same conclusion holds on
$B(\mathfrak{F})$. The converse is obvious,
so the proof is finished.
\end{proof}

\medskip
The point of our argument on the corona problem is to find such
an analytic disc $P(\xi)$ for a given $\phi$ in $\mathfrak{M}_1$
and for $f_1, f_2, \cdots, f_N$ in $H^\infty(\Delta)$.
Together with the corona theorem, Lemma \ref{lem2.3} shows
that the union of all homeomorphic analytic discs is dense in
$\mathfrak{M}(H^\infty)$.

\bigskip
\bigskip
\section{Approximation to representing measures}
\label{S3}

In this section we prepare three lemmas, which play an important role
in our argument. Recall that the restriction $A_1$ of
$H^\infty(\Delta)$ to $\mathfrak{M}_1$ is a uniform algebra,
whose Shilov boundary is $X\cap \mathfrak{M}_1$,
and also that an open-closed set $U=U(E)$ in $X$ satisfies that
$m(E)=\widehat{m}(U)> 0$ whenever $U\neq\emptyset$.
Here $\widehat{m}$ is the lifting of Lebesgue measure
$m$ to $X$.
Let $\phi$ be a homomorphism in $\mathfrak{M}_1$, and let
$S_\phi$ be the minimal support of the representing measure $\mu$
for $\phi$. Then $S_\phi$ is a compact subset of $X\cap \mathfrak{M}_1$.
For a simple function $s(\theta)$, the subset
$E=\{\theta\,; s(\theta)\neq 0 \}$ is called the
\textit{support} (\textit{carrier}) of $s(\theta)$.

\begin{lem}
\label{lem3.1}
Let $f_1, f_2, \cdots, f_N $ be in $H^\infty(\Delta)$, and let
$\varepsilon > 0$. Then there is a simple function
$s(\theta)\,=\, \sum_{i=1}^{\ell} a_i \chi_{E_i}(\theta)$ on
$\mathbf{T}$ with $a_i > 0$ such that the support
$E=\cup_{i=1}^\ell E_i$ is contained in
$(-\varepsilon, \varepsilon)$,
$\int_{E} s(\theta) dm(\theta)\,=\, 1$ and
\begin{equation}
\left \vert \phi(f_k) - \int_E f_k(\theta)s(\theta) dm(\theta) \right \vert
\;<\; \varepsilon, \quad k=1, 2, \cdots, N.
\label{eq3.1}
\end{equation}
Consequently, if $\{\varepsilon_j\}$ is a decreasing
sequence of positive numbers with $\varepsilon_j \to 0$, then
there is a sequence $\{s_j(\theta)\}$ of nonnegative simple
functions with supports $E^{(j)}$ such that
each $s_j(\theta)dm(\theta)$ is a probability measure
on $(-\varepsilon_j, \varepsilon_j)$ and
\begin{equation*}
\phi(f_k) \;=\;
\lim_{j \to \infty}\int_{E^{(j)}}\,
f_k(\theta)s_j(\theta)\,dm(\theta).
\end{equation*}
\end{lem}
\begin{proof}
Since $f_1, f_2, \cdots, f_N $ are continuous on $X$, we may
choose disjoint open-closed subsets $U_i = U(E_i)$ of $X$, $i=1,2,\cdots, \ell,$  such that $\cup_{i=1}^\ell U_i$ contains
$S_\phi$, $U_i\cap S_\phi \neq\emptyset,$ and
$$
\left \vert f_k(x) - f_k(y) \right \vert \;<\; \frac{\varepsilon}{2},\quad x,y \in
U_i\,,
$$
for $k=1,2,\cdots,N$. Since $S_\phi$ is the minimal
support, we observe that $\mu(U_i\cap S_\phi) > 0$.
We may assume the corresponding sets $E_i$ of $\mathbf{T}$ are
disjoint subsets of $(-\varepsilon, \varepsilon)$.
If we fix an $x_i$ in $U_i\cap S_\phi$, then

\begin{align*}
\left\vert \int_{S_\phi} f_k(x)\, d\mu(x) -\sum_{i=1}^{\ell}f_k(x_i)
\, \mu(U_i\cap S_\phi) \right \vert
&= \left \vert \sum_{i=1}^{\ell}\left( \int_{U_i\cap S_\phi} f_k(x)\, d\mu(x) - \int_{U_i\cap S_\phi} f_k(x_i)\, d\mu(x)\right) \right\vert \\
&\le \sum_{i=1}^{\ell} \int_{U_i\cap S_\phi}
\vert f_k(x) -f_k(x_i)\vert\, d\mu(x)\\
&< \frac{\varepsilon}{2}\,\sum_{i=1}^{\ell}\mu(U_i\cap S_\phi) \;=\; \frac{\varepsilon}{2}.
\end{align*}
Let
$$
a_i\,=\, \frac{\mu(U_i\cap S_\phi)}{m(E_i)} \qquad \text{and} \qquad s(\theta)\,=\, \sum_{i=1}^{\ell} a_i \chi_{E_i}(\theta)\,.
$$
Since $\vert f_k(\theta) - f_k(x_i)\vert < \varepsilon/2$ \,for $m$-a.e.
$\theta$ \,in $E_i$,
we obtain
$$
\left\vert \int_E f_k(\theta)s(\theta) dm(\theta) - \sum_{i=1}^{\ell}f_k(x_i)
a_i m(E_i) \right \vert \;<\; \frac{\varepsilon}{2}\,.
$$
Thus the simple function $s(\theta)$ on $\mathbf{T}$
satisfies the desired inequality \eqref{eq3.1}.
\end{proof}
We note that the above simple function $s(\theta)$ may
have the form $c \chi_{E}(\theta)$ with $c > 0$.
Indeed, if we choose each $U_i=U(E_i)$ satisfying that
$$
\frac{\mu(U_i\cap S_\phi)}{\mu(U_m\cap S_\phi)}\;=\;
\frac{m(E_i)}{m(E_m)} \quad \text{for}\quad i, m = 1, 2,\cdots, \ell,
$$
by cutting down suitably a part of $U_i\setminus S_\phi$, then
$s(\theta)$ has the form $s(\theta)=c \chi_{E}(\theta)$ with
$c = \mu(U_i\cap S_\phi)/m(E_i) > 0$.

\medskip
Let us now turn to a minor extension of Lemma \ref{lem3.1}.
\begin{lem}
\label{lem3.2}
Under the notation of Lemma \ref{lem3.1}, let \,$\mathfrak{F}^\sharp
= \{f_1, f_2, \cdots, f_N\,\}\cup \{1, z, B \,\}$ with
a fixed Blaschke product $B$ and the coordinate function $z$.
Then we have:

\smallskip
\noindent
{\rm (a)} \, There is a nonnegative simple function $s^\sharp(\theta)$
supported on a subset $F$ of $E$ such that $\int_F s^\sharp(\theta) dm(\theta)=1$,
and
\begin{equation}
\left \vert \phi(g) - \int_F g(\theta)s^\sharp(\theta) dm(\theta) \right \vert
\;<\; \varepsilon, \quad g \in \mathfrak{F}^\sharp.
\label{eq3.2}
\end{equation}

\noindent
{\rm (b)} \, Let $B(\mathfrak{F}^\sharp)$ be the uniformly closed subalgebra
of $H^\infty(\Delta)$ generated by $\mathfrak{F}^\sharp$.
Then there is a sequence $\{s^\sharp_j(\theta)\}$ of nonnegative simple
functions with supports $F^{(j)}$ such that
\begin{equation}
\label{eq3.3}
\phi(g) \;=\; \lim_{j \to \infty}\int_{F^{(j)}}\,
g(\theta)s^\sharp_j(\theta)dm(\theta),
\quad g \in B(\mathfrak{F}^\sharp).
\end{equation}
\end{lem}

\begin{proof}\,
{\rm (a)}\, Let $s(\theta)\,=\, \sum_{i=1}^{\ell} a_i \chi_{E_i}(\theta)$
be the simple function in Lemma \ref{lem3.1}. Since each $g$ in
$\{1, z, B \,\}$ is continuous on each $U(E_i)$,
there are disjoint subsets $F^{(i)}_j$ of $E_i$,
$j=1, 2, \cdots, m^{(i)}$ , such that
$S_\phi \cap U(F^{(i)}_j)\neq\emptyset$,
$$
\left\vert g(x) - g(y) \right\vert\; <\;\frac{\varepsilon}{2},
\quad x,y \in U(F^{(i)}_j),
$$
and the family $\{U(F^{(i)}_j) ; i=1,2,\cdots, \ell,\; j=1, 2,
\cdots, m^{(i)} \, \}$ forms a finite covering of $S_\phi$.
We then write $\{U(F_j)\,; \, k=1,2,\cdots, m \}$ for this family
$\{U(F_j^{(i)})\}$,
and put $b_j\,=\, \mu(U(F_j)\cap S_\phi)/m(F_j)$. By the same way
as in the proof of Lemma \ref{lem3.1}, we see that the simple function
$s^\sharp(\theta) \,=\, \sum_{j=1}^{m} b_j \chi_{F_j}(\theta)$ satisfies \eqref{eq3.2}
and the support $F=\cup_{j=1}^m F_j$ of $s^\sharp(\theta)$ is a
subset of $E=\cup_{i=1}^{\ell} E_i$.

\medskip
\noindent
{\rm(b)}\, Since, for $f, g$ in $\mathfrak{F}^\sharp$,
\begin{align*}
\left \vert (fg)(x) - (fg)(y) \right \vert
\;&\le\;\left \vert (fg)(x) - f(x)g(y) \right \vert +
\left \vert f(x)g(y) - (fg)(y) \right \vert \\
\;&<\; (\Vert f \Vert_\infty+\Vert g \Vert_\infty )\;\frac{\varepsilon}{2},
\quad x,y \in U(F_i),
\end{align*}
we observe that
\begin{equation*}
\left \vert \phi(fg) -
\int_F (fg)(\theta)s^\sharp(\theta) dm(\theta) \right \vert
\;<\; (\Vert f \Vert_\infty+\Vert g \Vert_\infty )\;\varepsilon.
\end{equation*}
Let $h$ be a finite product of functions in $\mathfrak{F}^\sharp$,
that is, $h = f_1^{m_1} f_2^{m_2} \cdots f_N^{m_N}\, z^m \, B^n$,
where $m_1, m_2, \cdots, m_N, m$ and $n$ are nonnegative integers.
Extending the above argument, we may choose a constant $C$, depending
on $f_k, B, m_k, m$ and $n$, such that $\vert h(x)-h(y)\vert\,<\,C\varepsilon/2$
for $x,y$ in $U(F_i)$ and
\begin{equation*}
\left \vert \phi(h) - \int_F h(\theta)s^\sharp(\theta) dm(\theta) \right \vert
\;<\; C \varepsilon.
\end{equation*}
Notice that $C$ never depends on $\varepsilon, F_i,$ and $s^\sharp(\theta)$.
Let $\{\varepsilon_j\}$ be a decreasing sequence of positive numbers
with $\varepsilon_j \to 0$.
It follows from {\rm(a)} and the argument above
that there is a sequence $\{s^\sharp_j(\theta)\}$ of simple functions
with decreasing supports $F^{(j)}$ such that
\begin{equation*}
\left \vert \phi(h) - \int_{F^{(j)}} \,
h(\theta)s^\sharp_j(\theta) dm(\theta) \right \vert
\;<\; C\varepsilon_j,
\end{equation*}
for the above finite product $h$ of functions in $\mathfrak{F}^\sharp$,
so the equality \eqref{eq3.3} holds for the function $h$.
Let $\mathfrak{F}_0^\sharp$ be the space of all finite linear
combinations of such product functions
$h = f_1^{m_1} f_2^{m_2} \cdots f_N^{m_N}\, z^m \, B^n$. Since $\phi$ and
the integrals by $s^\sharp_j(\theta) dm(\theta)$ are linear, \eqref{eq3.3}
holds for all functions in $\mathfrak{F}_0^\sharp$.
Observe that $\mathfrak{F}_0^\sharp$ is an algebra, in other words,
it is closed under the formation of multiplications.
Since $B(\mathfrak{F}^\sharp)$ is the uniform closure of
$\mathfrak{F}_0^\sharp$, the equation \eqref{eq3.3} extends to
$B(\mathfrak{F}^\sharp)$, which completes the proof.
\end{proof}

Since $\phi(fg)=\phi(f)\phi(g)$ for $f,g$ in $B(\mathfrak{F}^\sharp)$,
\eqref{eq3.3} shows that
\begin{equation*}
\left \vert  \int_{F^{(j)}} \,
f(\theta)s^\sharp_j(\theta) dm(\theta)\cdot \int_{F^{(j)}} \,
g(\theta)s^\sharp_j(\theta) dm(\theta) -
\int_{F^{(j)}} \,
(fg)(\theta)s^\sharp_j(\theta)\,dm(\theta)
 \right \vert
\end{equation*}
tends to $0$, as $j \to \infty$. Since the restriction $\phi \vert_B$ of
$\phi$ to $B(\mathfrak{F}^\sharp)$ is a homomorphism on the uniform
algebra, the property (b) of Lemma \ref{lem3.2} shows that, in a sense, $s^\sharp_j(\theta)\,dm(\theta)$ is close to a representing
measure for $\phi\vert_B$ on $B(\mathfrak{F}^\sharp)$.

\medskip
Let $s(\theta)\,=\, \sum_{i=1}^{\ell} a_i
\chi_{E_i}(\theta)$ be the simple function obtained
in Lemma \ref{lem3.1}. Since each term $a_i\chi_{E_i}(\theta)$ may be replaced by
$(a_i m(E_i)/m(F_i)) \chi_{F_i}(\theta)$ for a
subset $F_i$ of $E_i$ with $m(F_i)>0$, there are
many ways to represent such simple functions. Let
us always assume that $\mu(S_\phi\cap U(E_i)) >0$
and each $E_i$ is contained in either
$[0,\varepsilon)$ or $(-\varepsilon, 0]$. Hence, if $E_i$ is a subset of
$[0,\varepsilon)$, then $m(E_i\cap [0,\delta))>0$
for all $\delta>0$.
In order to discuss the relation between sequences
of such simple functions and Poisson kernels,
we need to choose certain analytic discs in $\mathfrak{M}_1$.
Since $s(\theta)dm(\theta)$ is a continuous probability
measure on $\mathbf{T}$, there are  $\alpha$ and $\beta$ with
$-\varepsilon<\alpha<\beta<\varepsilon$ such that
\begin{equation}
\label{eq3.4}
\int_{-\pi}^\alpha s(\theta)dm(\theta)\;= \:
\frac{1}{4}\; =\;\int_\beta^{\pi} s(\theta)dm(\theta).
\end{equation}
Observe that $\alpha$ and $\beta$ satisfy that
\begin{equation}
\label{eq3.5}
\begin{aligned}
-\varepsilon<\alpha<\beta \leq 0\;,&
\quad \text{if} \quad 0\leq \mu(S_\phi\cap U([0,\varepsilon)))\leq 1/4,\\
-\varepsilon<\alpha \leq 0 <\beta \;,&
\quad \text{if} \quad 1/4 < \mu(S_\phi\cap U([0,\varepsilon)))\leq 3/4,\\
0 < \alpha<\beta <\varepsilon\;,&
\quad \text{if} \quad 3/4 < \mu(S_\phi\cap U([0,\varepsilon)))\leq 1.
\end{aligned}
\end{equation}
By our assumption on $E_i$, we also observe that if
$\mu(S_\phi\cap U([0,\varepsilon)))=3/4$ or $1/4$, then
$\alpha=0$ or $\beta=0$, respectively.
It is sometimes useful to modify $\alpha$ and $\beta$ suitably.

\medskip
Let $\alpha$ and $\beta$ be as above, and let $C$ be
the circular arc from $e^{i\alpha}$ to $e^{i\beta}$ orthogonal
to the unit circle $\mathbf{T}$ lying in $\Delta$, and
put $c$ to the point in $C$ meeting the line
$\ell(t)=t\,e^{i(\alpha+\beta)/2}$ with $0\leq t \leq 1$.
We call $c$ the {\it mid-point} of the arc $C$. If $L_c(\zeta)=(\zeta+c)/(1+\bar{c}\,\zeta)$ as in \eqref{eq2.2},
then $L^{-1}_c(z)=(z-c)/(1-\bar{c}z)$.

\medskip
Let us turn to certain properties of M\"{o}bius transformations
to investigate the desired analytic discs.
The next lemma is so fundamental that we omit the proof\,:

\begin{lem}
\label{lem3.3}
Under the above hypotheses, $L^{-1}_c$ maps the closed unit disc $\overline{\Delta}=\Delta \cup \mathbf{T}$ onto itself such that
$L^{-1}_c(c)=0$ and $L^{-1}_c(e^{i\alpha})= -L^{-1}_c(e^{i\beta})$,
that is, symmetric with respect to $0$. Moreover, if $g$ is
a function in $L^1(\mathbf{T})$, then we have
\begin{equation}
\label{eq3.6}
\int_{L^{-1}_c(A)} (g\circ L_c)(\theta)\,\vert \,(L_c)'(\theta)\,\vert
\,dm(\theta)\;=\;\int_A g(\theta)\,dm(\theta)
\end{equation}
for all Borel sets $A$ in $\mathbf{T}$. Particularly, if we set
$u(\theta)=(s\circ L_c)(\theta)\,\vert\, (L_c)'(\theta)\,\vert$,
then
\begin{equation*}
\int_{L^{-1}_c(E)} (f\circ L_c)(\theta)u(\theta)\,dm(\theta)\;=\;
\int_E f(\theta)s(\theta)\,dm(\theta),\quad f \in H^\infty(\Delta),
\end{equation*}
where $E$ is the support of $s(\theta)$.
\end{lem}

\bigskip
\bigskip
\section{Construction of auxiliary Blaschke products}
\label{S4}

In this section we derive certain Blaschke products from
given ones, which play an important role in the proof of
Theorem. For a Blaschke product $B$, we denote by $Z(B)$ the
set of all zeros of $B$ repeated multiplicity for each zero.
Let us show some elementary properties related to $Z(B)$.

\begin{lem}
\label{lem4.1}
Let $\/0<\eta<1$ and let $\/\varepsilon > 0$. Then there is a
$\delta=\delta(\varepsilon,\eta) > 0$ such that, for any
Blaschke product $B$ with $Z(B)=\{z_k\}$,
the condition,
$$
\sum_{k=1}^\infty (1-\vert z_k\vert)\; <\;\delta,
$$
on $Z(B)$ implies that
\begin{equation}
\label{eq4.1}
\vert B(z)\vert\; >\; 1-\varepsilon, \quad
\text{for} \quad \vert z \vert \;\le\; \eta\,.
\end{equation}
\end{lem}

\begin{proof}
\, When $\vert z\vert \le \eta$, we observe that
\begin{equation}
\label{eq4.2}
1 - \left \vert \frac{z_k - z}{1 - \overline{z_k}z} \right \vert \leq
\left \vert 1 - \frac{z_k - z}{1 - \overline{z_k}z}\cdot
\frac{\vert z_k\vert}{z_k} \right \vert
\leq \frac{1+\eta}{1-\eta}\left(1-\vert z_k\vert\right).
\end{equation}
Since
$$
-\log t \le \frac{-2\log a}{1-a^2}(1-t)
\le (1+2\log\frac{1}{a})(1-t)
$$
is valid for $a^2<t<1$ (see \cite[Chapter VII, Lemma 1.2]{Ga}),
\begin{align*}
-\log \vert B(z)\vert\;&=-\sum_{k=1}^\infty \log \left \vert
\frac{z_k - z}{1 - \overline{z_k}z} \right \vert \\
&\le C_1 \sum_{k=1}^\infty \left (1 - \left \vert
\frac{z_k - z}{1 - \overline{z_k}z} \right \vert\right ) \;\le\;
C_2\sum_{k=1}^\infty \left(1-\vert z_k\vert \right)
\end{align*}
for some constants $C_1$ and $C_2$. Then we have
\begin{equation*}
\vert B(z)\vert\; \ge \;e^{-C_2\delta}\;>\; 1-\varepsilon \quad \text{for} \quad \vert z \vert \;\le\; \eta,
\end{equation*}
with sufficiently small $\delta=\delta(\varepsilon,\eta) > 0$.
\end{proof}
Recall that if $L_c(\zeta)= (\zeta + c)/(1 + \overline{c}\zeta)$
with $\vert c \vert < 1$, then
$L_c^{-1}(z)= (z -c)/(1 - \overline{c}z)$.

\begin{lem}
\label{lem4.2}
Let $L_c$ and $L_c^{-1}$ be as above, and let $B$ be a Blaschke
product with $Z(B)=\{z_k\}$. Then $B\circ L_c$ itself is the Blaschke
product with $Z(B\circ L_c)=\{L_c^{-1}(z_k)\}$. In particular, for a given
$\delta>0$, there is an $N$ such that
$$
\sum_{k=N}^\infty (1-\vert L_c^{-1}(z_k)\vert)\; <\;\delta\,.
$$
\end{lem}

\begin{proof}
By the similar way as in the proof of (\ref{eq4.1}) we observe that
\begin{equation*}
1 - \left \vert
\frac{z_k - c}{1 - \overline{c}z_k} \right \vert
= 1 - \left \vert
\frac{z_k - c}{1 - c\overline{z_k}}\cdot
\frac{\vert z_k\vert}{z_k} \right \vert
\leq
\frac{1+\vert c \vert}{1-\vert c \vert}(1-\vert z_k\vert).
\end{equation*}
So $ \zeta_k=(z_k - c)/(1 - \overline{c}z_k)$ is a Blaschke sequence,
meaning that $\sum_{k=1}^\infty (1-\vert \zeta_k\vert)\; <\;\infty.$
On the other hand, if $S(\zeta)$ is a nonconstant singular function,
then so is $S\circ L_c^{-1}(z)$, because it has no zeros on $\Delta$.
This shows that the inner function $B\circ L_c$ cannot have a singular
factor, so $B\circ L_c$ is the Blaschke product with
$Z(B\circ L_c)=\{\zeta_k\}$, as desired.
\end{proof}
We notice that $B\circ L_c(z)$ has the form of original Blaschke product
multiplied a constant of modulus one, while $L_c\circ B (z)$ may happen
to be a singular function by the Frostman theorem
\cite[Chapter II, Theorem 6.4]{Ga}.

\medskip
Let $\{c_n\}$ be a sequence in $\Delta$ with $c_n \to 1$. For an
$\eta >0$, we denote by $K(c_n,\eta)$ the noneuclidean disc
\begin{equation*}
K(c_n,\eta)\;=\;\left\{z\in \Delta ; \rho(z,c_n)= \left \vert
\frac{z - c_n}{1 - \overline{c_n}z} \right \vert\, < \,\eta \right\}
\;=\; L_n(\{\vert \zeta \vert < \eta \}),
\end{equation*}
where $L_n $ is the map on $\Delta$ in \eqref{eq2.2}.
Then $K(c_n,\eta)$ is the euclidean disc with center
$a_n = (1 - \eta^2)c_n/(1 - \eta^2\vert c_n\vert^2)$ and radius
$r_n = \eta(1 - \vert c_n\vert^2)/(1 - \eta^2\vert c_n\vert^2)$
(see \cite[Chapter I, \S 1]{Ga}). Observe that $a_n  \to 1$
and $r_n \to 0$, as $c_n \to 1$.

Suppose $\phi$ is a homomorphism in the fiber $\mathfrak{M}_1$
outside the Shilov boundary $X$ of $H^\infty(\Delta)$. Then Newman's
theorem assures the existence of a Blaschke product $B_0$  vanishing at
$\phi$ (see \cite [Chapter 10]{H} or \cite [Chapter V, Theoem 2.2]{Ga}).
By modifying $B_0$ suitably, we construct a certain Blaschke
product $B$ with $B(\phi) = 0$ such that,
for a subsequence $\{c_{n_j}\}$ of $\{c_n\}$, the function,
$\lim_{j\to\infty}\,B\circ L_{n_j}(\zeta)$, generates
the disc algebra $A(\Delta)$, the uniform algebra on $\Delta$
generated by $G(\zeta)=\zeta$.

Let $0<\ell <1$, and let $[s,t)$ be the interval with
$\ell \le s < t \le 1$.
Then $S[s,t)$ denotes the sector
\begin{equation*}
S[s,t)\;=\;\left \{\, re^{i\theta} \, ; \, r \in [s,t),\;
\vert \theta\vert \le \frac{1-\ell}{2}\, \right \}\,.
\end{equation*}
Since each Blaschke product with zeros outside $S[s,1)$ is
continuous on $\{ re^{i\theta} \,; \,\vert \theta\vert < (1-\ell)/2\}$,
we assume $Z(B_0)$ is contained in $S[\ell,1)$, for the above $B_0$.
Notice that $S[\ell,t) \cap Z(B_0)$ is always finite whenever
$\ell < t < 1$, and that each Blaschke product with zeros
$S[t,1) \cap Z(B_0)$ has always the value $0$ at $\phi$.

\begin{lem}
\label{lem4.3}
Let $\phi, \{c_n\}$ and $L_n$ be as above. Then we may choose
a Blaschke product $B$ with $B(\phi) = 0$ such that,
for some subsequence $\{c_{n_j}\}$ of $\{c_n\}$, $B(c_{n_j}) = 0$
and $B\circ L _{n_j}(\zeta)$ converges uniformly to $G(\zeta)=\zeta$
on compact subsets of $\Delta$.
\end{lem}

\begin{proof}
Let $\{\varepsilon_n\}$ be a decreasing sequence of positive numbers
with $\varepsilon_n \to 0$, and let $\{\eta_n\}$ be an increasing
sequence of positive numbers with $\eta_n \to 1$.
If we put $Z(B_0)\,=\,\{z_k\}$ for the above $B_0$,
then $B_0 \circ L_n$ is a Blaschke product with
$Z(B_0 \circ L_n) = \{L_n^{-1}(z_k)\}$ by Lemma \ref{lem4.2}.
Observe that $\vert L_n^{-1}(z_k)\vert \to 1$, as $c_n \to 1$.

Let $s_1=\ell$ and $r_1 = 1- 2(1-s_1)/3= (2s_1+1)/3$. It follows from
Lemma \ref{lem4.1} that there is a $\delta_1>0 $ for which \eqref{eq4.1}
holds with $\varepsilon_1$ and $\eta_1$.
Since $S[\ell,r_1)\cap Z(B_0)$ is finite, there is a $c_{n_1}$ in $\{c_n\}$
such that
\begin{equation*}
\sum_{S[\ell,r_1)\cap Z(B_0)\ni z_k}\;\left(1-
\vert L_{n_1}^{-1}(z_k)\vert \right)
\;<\; \frac{\delta_1}{2}.
\end{equation*}
We fix such a $c_{n_1}$ in $\{c_n\}$.
Observe that, for any $\rho_1 > 0$, there is a $\rho_2 > 0$
such that $\vert L_{n_1}^{-1}(z_k)\vert > \rho_1$ whenever
$\vert z_k \vert > \rho_2 $. Hence there is an $s_2$ with
$r_1<s_2<1$ such that
\begin{equation*}
\sum_{S[s_2,1)\cap Z(B_0)\ni z_k}\;
\left(1-\vert L_{n_1}^{-1}(z_k)\vert \right)
\;<\; \frac{\delta_1}{2},
\end{equation*}
which is an infinite sum. Let $B^{(1)}$ be the
Blaschke product with zeros $z_k$
in $S[\ell,r_1)\cup S[s_2,1)$, that is, $Z(B^{(1)}) = Z(B_0)\cap (S[\ell,r_1)\cup S[s_2,1))$.
It follows from Lemmas \ref{lem4.1} and \ref{lem4.2} that
$$
\vert B^{(1)}\circ L_{n_1}(\zeta)\vert\; >\; 1-\varepsilon_1
\quad \text{for}\quad \vert \zeta \vert \;\le \;\eta_1.
$$
We then put $r_2 = (2s_2+1)/3$. By repetitions of the process
on ad infinitum, we choose the sequences $\{s_j\}$, $\{c_{n_j}\}$
and $\{B^{(j)}\}$ satisfying that, for
\begin{equation*}
\ell = s_1<r_1<s_2<r_2<\cdots <s_j<r_j<\cdots <1
\end{equation*}
with $r_j=(2s_j+1)/3$, the zero-set of $B^{(j)}$ is
$Z(B_0)\cap (S[\ell,r_j)\cup S[s_{j+1},1))$, and the
Blaschke product $ B^{(j)}\circ L_{n_j}$
satisfies
\begin{equation}
\label{eq4.3}
\vert B^{(j)}\circ L_{n_j}(\zeta)\vert\; >\; 1-\varepsilon_j
\quad \text{for}\quad \vert \zeta \vert \;\le\; \eta_j.
\end{equation}
Notice that if a Blaschke product has the zero-set contained
in $Z(B^{(j)})$, then it satisfies the same inequality \eqref{eq4.3},
and also notice that $\cup_{j=1}^{\infty} \,S[s_j,r_j)$,
$\cup_{j=1}^{\infty} \,S[r_{2j-1},s_{2j})$ and
$\cup_{j=1}^{\infty} \,S[r_{2j},s_{2j+1})$ are disjoint one another.
We then consider the three Blaschke products $B_1, B_2$
and $B_3$ whose zero-sets are given by
\begin{align*}
Z(B_1)\;&=\; Z(B_0)\cap \left(S[s_1,r_1)\cup S[s_2,r_2)
\cup \cdots\cup S[s_j,r_j)\cup \cdots \right), \\
Z(B_2)\;&=\; Z(B_0)\cap \left(S[r_1,s_2)\cup S[r_3,s_4)
\cup \cdots\cup S[r_{2j-1},s_{2j})\cup\cdots \right), \\
Z(B_3)\;&=\; Z(B_0)\cap \left(S[r_2,s_3)\cup S[r_4,s_5)
\cup \cdots\cup S[r_{2j},s_{2j+1})\cup\cdots \right),
\end{align*}
respectively. Since $ B_0(\phi) = 0$ and $B_0=B_1B_2B_3$,
we observe that
either $B_1 B_2$ or $B_1 B_3$ vanishes at $\phi$. Now let
us assume that $(B_1B_2)(\phi) = 0$, because the other case
is dealt with similarly. Since we see easily
$$
Z(B_1B_2)\;=\;\bigcap_{j=1}^{\infty}\,\left[Z(B_0)\cap \left(S[\ell,r_{2j})\cup S[s_{2j+1},1)\right)\right]
\;=\; \bigcap_{j=1}^{\infty}\,Z(B^{(2j)}),
$$
it follows that
$$
\vert (B_1B_2)\circ L_{n_{2j}}(\zeta)\vert\; >\; 1-\varepsilon_{2j}
 \quad \text{for}\quad \vert \zeta \vert \;\le \;\eta_{2j}\,,
$$
for $j=1, 2, \cdots.$
Replacing $\{c_{2j}\}$ with its suitable subsequence and
multiplying some unimodular constant, we assume that
$(B_1B_2)\circ L_{n_{2j}}(\zeta)$ converges
uniformly to the constant $1$ on compact subsets of $\Delta$.
We may also assume $\{c_{2j}\}$
is a zero-set of a thin Blaschke product $B_4$ such that
$B_4\circ L_{n_{2j+1}}(\zeta)$ converges uniformly to
$G(\zeta)=\zeta$ on compact subsets of $\Delta$.
Let us write $\{c_{n_j}\}$ for $\{c_{n_{2j}}\}$.
Then the Blaschke product $B=B_1B_2B_4$ satisfies the desired
properties.
\end{proof}

\medskip
Let us make some remarks on Lemma 4.3. When $B_1(\phi)=0$,
we may replace $B_1B_2B_4$ with $B_1B_4$ in the argument above.
It is rather easy to find such a $B$ whenever there exists an
interpolating Blaschke product $B_0$ with $B_0(\phi)=0$
(compare with \cite[Chapter IX, Lemma 3.3]{Ga}).
We next consider the relation between the boundary value
functions of $B\circ L _{n_j}(\zeta)$ and $G(\zeta)$.
Usually, $B\circ L _{n_j}(\zeta)$ has a lot of zeros near to
the boundary $\mathbf{T}$ of $\Delta$, in contrast to $G(\zeta)$.

\begin{lem}
\label{lem4.4}
Under the notation of Lemma \ref{lem4.3}, we may choose a subsequence
$\{c_{n_k}\}$ of $\{c_{n_j}\}$ such that $B\circ L _{n_k}(\theta)$
converges to $G(\theta)= e^{i\theta}$, $m$-a.e. $\theta$ in
$\mathbf{T}$, consequently, $B\circ L _{n_k}(x)$
converges to $G(x)$, $\widehat{m}$-a.e. $x$ in
$X$, where $\widehat{m}$ denotes the lifting of $m$ to $X$.
\end{lem}

\begin{proof}
Since each Blaschke product $B\circ L _{n_j}$ lies in
$L^2(\mathbf{T})$, its Taylor series shows that
$$
B\circ L _{n_j}(\theta) \,=\, \sum_{k=0}^{\infty}\, a_k^{(j)}
e^{ik\theta}\quad \text{and} \quad
\Vert B\circ L _{n_j}\Vert^2_2\,=\,\sum_{k=0}^{\infty}\, \vert
a_k^{(j)}\vert^2 \,=\,1\,.
$$
It follows from Lemma \ref{lem4.3} and the Weierstrass theorem
of double series that $a_1^{(j)} = (B\circ L _{n_j})'(0)$
tends to $G'(0)=1$. Since $a_0^{(j)}=B\circ L _{n_j}(0)
= B(c_{n_j})=0$, we have
$$
\Vert B\circ L _{n_j} - G \Vert^2_2 \,=\, \vert a_1^{(j)} - 1\vert^2 + \sum_{k=2}^{\infty}\, \vert
a_k^{(j)}\vert^2 \rightarrow 0, \quad j \to \infty.
$$
Thus, we may choose a subsequence $\{c_{n_k}\}$ of $\{c_{n_j}\}$
such that $B\circ L _{n_k}(\theta)$
converges to $G(\theta)= e^{i\theta}$, $m$-a.e. $\theta$ in
$\mathbf{T}$.
\end{proof}

\bigskip
\bigskip
\section{Existence of desired sequences}
\label{S5}
Let $\phi$ be a homomorphism in
$\mathfrak{M}(H^\infty)\setminus \Delta$,
and let $f_1, f_2, \cdots, f_N $ be functions in
$H^\infty(\Delta)$. What should be shown is the existence of a
sequence $\{\zeta_j\}$ in $\Delta$ for which $\lim_{j\to \infty} f_k(\zeta_j)
= f_k(\phi)$ for $k=1,2,\cdots, N$. By Lemma \ref{lem2.1}, it suffices to
consider the case where $\phi$ lies in $\mathfrak{M}_1\setminus X$.
So the representing measure $\mu$ for $\phi$ is a continuous
measure, and its minimal support $S_\phi$ is a compact subset of
$\mathfrak{M}_1\cap X$. Then we know that there exists
a Blaschke product vanishing at $\phi$ in $\mathfrak{M}(H^\infty)$.

Let $\{\varepsilon_n\}$ be a decreasing sequence of positive numbers with
$\varepsilon_n \to 0$, and let $s_n(\theta)\,=\,\sum_{i=1}^{\ell^{(n)}}
a_i^{(n)} \chi_{E_i^{(n)}}(\theta)$
denote the simple function determined by Lemma \ref{lem3.1},
of which support $E^{(n)} = \cup_{i=1}^{\ell^{(n)}}\, E_i^{(n)}$
is contained in $(-\varepsilon_n, \varepsilon_n)$.
Then there are
$\alpha_n$ and $\beta_n$ with $-\varepsilon_n < \alpha_n
< \beta_n < \varepsilon_n$ satisfying the properties of
\eqref{eq3.4} and \eqref{eq3.5}.
Let $c_n=t_n\,e^{i(\alpha_n+\beta_n)/2}$ in $\Delta$ be the mid-point
of the circular arc $C_n$ from $e^{i\alpha_n}$ to $e^{i\beta_n}$ orthogonal to
$\mathbf{T}$. Observe that $c_n \to 1$, as $n\to\infty$. We now
choose a thin interpolating subsequence $\{c_{n_j}\}$ of $\{c_n\}$
for which a Blaschke product $B$ with $B(\phi)=0$ has the
property of Lemma \ref{lem4.3}. For simplicity of notation, we write $\{\varepsilon_j\}, \{s_j(\theta)\}$ and $\{c_j\}$ for
$\{\varepsilon_{n_j}\}, \{s_{n_j}(\theta)\}$ and $\{c_{n_j}\}$,
respectively.

Let $\mathfrak{F}^\sharp\,= \{f_1, f_2, \cdots, f_N\,\}\cup \{1, z, B \,\}$
with the Blaschke product $B$ above, and let $B(\mathfrak{F}^\sharp)$
be the uniformly closed subalgebra of $H^\infty(\Delta)$ generated by
$\mathfrak{F}^\sharp$ as in Section 3. Let $L_j$ be the map $L_j(\zeta)=(\zeta+c_j)/(1+\overline{c_j}\zeta)$ as usual.
Replacing $\{c_j\}$ with a suitable subsequence
if necessary, we may assume, by normal families argument, that
$f_k\circ L_j(\zeta)$ converges uniformly
to $F_k(\zeta)$ on compact subsets of $\Delta$, as $j \to \infty$.
Since each function in $\{1, z, B \,\}$ also has the same property,
it holds on all $B(\mathfrak{F}^\sharp)$, that is, if $h$ is in
$B(\mathfrak{F}^\sharp)$, then there is an $H$ in $H^\infty(\Delta)$ such that
\begin{equation}
\label{eq5.1}
\lim_{j\to\infty}\, h\circ L_j(\zeta)\;=\; H(\zeta)
\end{equation}
uniformly on compact subsets of $\Delta$.
Modifying the equation, we obtain the following:

\begin{lem}
\label{lem5.1}
Let $B$ be the Blaschke product with the property of Lemma \ref{lem4.3}.
Under the notation of \eqref{eq5.1}, here are some properties:

{\rm (a)} \,We have
\begin{equation*}
\lim_{j\to\infty}\, h\circ L_j(B\circ L_j(\zeta))
\;=\;
H(\zeta)
\;=\;
\lim_{j\to\infty}\, (H\circ B)\circ L_j(\zeta)
\end{equation*}
uniformly on compact subsets of $\Delta$. Moreover, we also observe that
\begin{equation*}
\lim_{j\to\infty}\, \left\vert \,h\circ L_j(\zeta) -
(H\circ B)\circ L_j(\zeta)\, \right\vert
\;=\;0
\end{equation*}
uniformly on compact subsets of $\Delta$.

\medskip
{\rm (b)} \,Regarding $H$ as a function in $L^\infty(\mathbf{T})$,
we may choose a sequence $\{p_n\}$ of polynomials with \;
$\Vert p_n \Vert_\infty \leq \Vert H \Vert_\infty$ such that
$\lim_{n\to\infty}\, \Vert \, p_n - H \,\Vert_1 \,=\, 0$,
consequently,
\begin{equation*}
\lim_{n\to\infty}\, p_n(\zeta)\,=\,H(\zeta)
\end{equation*}
uniformly on compact subsets of $\Delta$.
\end{lem}

\begin{proof}
The first equation of (a) is a direct consequence of \eqref{eq5.1} and the
property of $B$. The second one follows from the inequality
\begin{equation*}
\left\vert \,h\circ L_j(\zeta) -
(H\circ B)\circ L_j(\zeta)\, \right\vert
\; \leq \;
 \left\vert \,h\circ L_j(\zeta) - H(\zeta)\, \right\vert +
 \left\vert H(\zeta) - (H\circ B)\circ L_j(\zeta)\,\right\vert.
\end{equation*}
On the other hand, it is known that the Cesaro means $p_n$ of the
Fourier (Taylor) series of $H$ converges to $H$ in
$L^1(\mathbf{T})$ (see \cite[16p]{H}). Since $H$ is bounded on $\Delta$,
we observe that
$\Vert p_n \Vert_\infty \leq \Vert H \Vert_\infty$. Hence, the
conclusion of (b) follows from Poisson integral formulas for
analytic functions on $\Delta$.
\end{proof}

\medskip
As we shall see later, $H(\zeta)$ is identified with the
restriction of $h$ or $H\circ B$ in $H^\infty(\Delta)$
to a suitable homeomorphic part in $\mathfrak{M}_1$
(see the remarks below).

\medskip
Moreover, it follows from Lemma \ref{lem3.2} that there
is a sequence $\{s^\sharp_j(\theta)\}$ of simple functions with
supports $F^{(j)}$ satisfying \eqref{eq3.3} on $B(\mathfrak{F}^\sharp)$.
We especially have
\begin{equation}
\label{eq5.2}
\phi(B)^n \;=\;\phi(B^n)\;=\; \lim_{j \to \infty}\int_{F^{(j)}}\,
B(\theta)^n \,s^\sharp_j(\theta)\,dm(\theta)\,,\quad
n = 0, 1, 2, \cdots,
\end{equation}
which vanishes for each $n \geq 1$.

\medskip
Let $M(X)$ be the space of all regular complex Borel measures on
$X=\mathfrak{M}(L^\infty)$. Then the Riesz representation theorem
shows that $M(X)$ is identified with the dual space of
$C(X)=L^\infty(\mathbf{T})$. Since each $s^\sharp_j(\theta)\, dm(\theta)$
extends to a probability measure $s^\sharp_j(x)\,d\widehat{m}(x)$
on $X$, we may choose an adherent point $\widehat{\nu}$ of
$\{s^\sharp_j(x)\,d\widehat{m}(x)\}$ in $M(X)$ in the weak*-topology.
Then $\widehat{\nu}$ is a representing measure for the homomorphism
which is the restriction of $\phi$ to $B(\mathfrak{F}^\sharp)$.
Indeed, we have
\begin{equation}
\label{eq5.3}
\begin{aligned}
h(\phi)\;&=\; \lim_{j\to \infty} \int_{-\pi}^{\pi}
h(\theta) s^\sharp_j(\theta)\, dm(\theta)\\
&=\;\lim_{j\to \infty} \int_X
h(x) s^\sharp_j(x)\, d\widehat{m}(x)\\
&=\;\int_X\,h(x)\,d\widehat{\nu}(x)
\end{aligned}
\end{equation}
for all $h$ in $B(\mathfrak{F}^\sharp)$, because Lemma \ref{lem3.2}
assures the existence of above limits on $B(\mathfrak{F}^\sharp)$.
It would be a crucial point that the sequence of
$\{s^\sharp_j(\theta)dm(\theta)\}$ may be replaced by the one of
Poisson integrals for functions in $B(\mathfrak{F}^\sharp)$.

\medskip
Let $\mathbf{N}^\infty = \mathbf{N} \cup \{\infty\}$ be the
one-point compactification of all positive integers $\mathbf{N}$,
and let $H^\infty(\Delta\times\mathbf{N}^\infty)$ be the space
of all bounded sequences $g(\zeta,j)$, for $j$ in $\mathbf{N}^\infty$,
of analytic functions on $\Delta$ such that $g(\zeta,j)$
converges to $g(\zeta,\infty)$ uniformly on compact subsets of
$\Delta$. It then follows from Fatou's theorem that each function
$\zeta \to g(\zeta,j)$ extends to the closed unit disc
$\overline{\Delta} = \Delta\cup \mathbf{T}$, of which the boundary
function is denoted by $g(\theta,j)$. As a sequence $\{g(\theta,j)\}$
in $L^\infty(\mathbf{T})=L^1(\mathbf{T})^\ast$, the dual space of
$L^1(\mathbf{T})$, it is easy to see that $g(\theta,j)$ converges
to $g(\theta,\infty)$ in the weak*-topology
( compare with Lemma \ref{lem4.4}). Let $u_j(\theta)
= s_j^\sharp\circ L_j(\theta) \,\vert (L_j)'(\theta)\,\vert$.
Then Lemma \ref{lem3.3} and
\eqref{eq5.3} assure that the following equation
\begin{equation}
h(\phi)\; =\; \lim_{j\to \infty} \int_{-\pi}^{\pi}
h(\theta)\,s^\sharp_j(\theta) dm(\theta)
\;=\;\lim_{j\to \infty} \int_{-\pi}^{\pi}
h\circ L_j(\theta)\, u_j(\theta)\, dm(\theta)
\label{eq5.4}
\end{equation}
holds on $B(\mathfrak{F}^\sharp)$. We have to represent the
measure $\widehat{\nu}$ in \eqref{eq5.3} as the one on
$\mathbf{T}\times \{\infty\}$.

\bigskip
We are now ready for the proof of Theorem.

\medskip
\noindent
{\itshape Proof of Theorem.\;}
Let us define a closed subalgebra $(SB)(\mathfrak{F}^\sharp)$
of $H^\infty(\Delta\times\mathbf{N}^\infty)$. With the above notation,
for each function $h$ in $B(\mathfrak{F}^\sharp)$, we put
\begin{equation}
g(\zeta,j)=
\begin{cases}
\; h\circ L_j(\zeta)\,, & \;j \in \mathbf{N}, \\
\; H(\zeta)\,, & \;j = \infty,
\end{cases}
\label{eq5.5}
\end{equation}
in the equation \eqref{eq5.1}. Then $g(\zeta,j)$ is an element of
$H^\infty(\Delta\times\mathbf{N}^\infty)$. We set $(SB)(\mathfrak{F}^\sharp)$
to be the space of all such functions $g(\zeta,j)$ by \eqref{eq5.1}.
Observe that $(SB)(\mathfrak{F}^\sharp)$ is a uniformly closed subalgebra
of $H^\infty(\Delta\times\mathbf{N}^\infty)$. We may regard
$u_j(\theta)dm(\theta)=
s_j^\sharp\circ L_j(\theta) \,\vert (L_j)'(\theta)\,\vert dm(\theta)$
as a probability measure on $\mathbf{T}\times \mathbf{N}^\infty$
concentrated on $\mathbf{T}\times \{j\}$.

Let $\mathbf{H}(\Delta)$ denote the restriction of $(SB)(\mathfrak{F}^\sharp)$
to $\Delta\times \{\infty\}$. Then if $H$ lies in $\mathbf{H}(\Delta)$, then
there is an $h$ in $B(\mathfrak{F}^\sharp)$ such that
$H(\zeta) = \lim_{j\to\infty}\, h\circ L_j(\zeta)$.
We claim that the value $h(\phi)$ may be represented as
\begin{equation}
h(\phi)\; =\; H(0) \; =\; \int_{-\pi}^{\pi} H(\theta)\,dm(\theta)\,,
\label{eq5.6}
\end{equation}
where $H(\theta)$ is the boundary value function of $H(\zeta)$.
Indeed, recall that the Blaschke product $B$ in $B(\mathfrak{F}^\sharp)$
satisfies that $B\circ L_j(\zeta)$ converges to $G(\zeta)=\zeta$
uniformly on compact subsets of $\Delta$, so does
$p(B\circ L_j)(\zeta)$ to $p(G)(\zeta)=p(\zeta)$, for each
polynomial $p$ on $\Delta$. Since $p(B)$ lies in $B(\mathfrak{F}^\sharp)$,
the function,
\begin{equation*}
g(\zeta,j)=
\begin{cases}
\; p(B\circ L_j)(\zeta)\,, & \;j \in \mathbf{N}, \\
\; p(\zeta)\,, & \;j = \infty,
\end{cases}
\end{equation*}
is an element of $(SB)(\mathfrak{F}^\sharp)$,
the space $\mathbf{H}(\Delta)$ contains all polynomials.
Since the disc algebra $A(\Delta)$ is the uniform closure
of polynomials, the subalgebra $\mathbf{H}(\Delta)$ of
$H^\infty(\Delta)$ also contains $A(\Delta)$. Recall that we have
$B(\phi) = 0 = G(0)$ by the property of Lemma \ref{lem4.3}.
Then \eqref{eq5.2} and \eqref{eq5.4} show that
$P\circ B(\phi) = P(G(0)) = P(0)$.

On the other hand, for each $H$ in $\mathbf{H}(\Delta)$,
there is a sequence $\{p_n\}$ of polynomials with the
property (b) of Lemma \ref{lem5.1}, so $p_n(0) \to H(0)$.
Since
\begin{align*}
\lim_{j\to \infty} \int_{-\pi}^{\pi}
\left(h\circ L_j(\theta)-p_n(B\circ L_j)(\theta)\,\right)\,
u_j(\theta)\, dm(\theta) \;&=\;h(\phi) - p_n\circ B(\phi)\\
&=\;h(\phi) - p_n(0),
\end{align*}
we have
\begin{equation*}
\vert h(\phi) - H(0) \vert \leq \vert h(\phi) - p_n\circ B(\phi) \vert +
\vert p_n(0) - H(0) \vert \rightarrow 0\,,
\end{equation*}
as $n\to \infty$. Thus \eqref{eq5.6} holds for all
$h$ in $B(\mathfrak{F}^\sharp)$.

Let $\zeta_j = L_j(0) = c_j$. We show that the sequence $\{\zeta_j \}$
in $\Delta$ satisfies the property \eqref{eq1.2} for $f_1, f_2, \cdots, f_N $. Indeed, for all $h$ in $B(\mathfrak{F}^\sharp)$,
Lemma \ref{lem5.1} implies that
\begin{equation*}
H(\zeta)\;=\; \lim_{j\to\infty}\,h\circ L_j (B\circ L_j(\zeta)),
\end{equation*}
uniformly on compact subsets of $\Delta$. Hence, since $B\circ L_j(0)=0$,
it follows from \eqref{eq5.6} that
\begin{equation*}
h(\phi)\;=\;H(0)\;=\;\lim_{j\to\infty}\,h\circ L_j(0)
\;=\;\lim_{j\to\infty}\,h(\zeta_j).
\end{equation*}
Since each $f_k$ lies in $B(\mathfrak{F}^\sharp)$, we obtain
$\lim_{j\to \infty} f_k(\zeta_j) = f_k(\phi)$. Therefore,
by virtue of Lemma \ref{lem2.3}, the property \eqref{eq1.2} holds
on the uniform algebra $B(\mathfrak{F})$ generated by
a countable family $\mathfrak{F}$ of $H^\infty(\Delta)$,
as desired.
\hfill $\square$

\medskip

Some remarks are in order on the above proof.
Let us explain the relation between $\Delta\times \{\infty\}$ and
a homeomorphic part in $\mathfrak{M}_1\setminus \Delta$, which
reveals the structure of $\mathbf{H}(\Delta)$. Since $\{c_j\}$ is
a thin interpolating sequence in $\Delta$, any adherent point $\xi$
of $\{c_j\}$ lies in the homeomorphic part $P(\xi)$ in
$\mathfrak{M}_1\setminus \Delta$. Let $L= \lim_\beta L_\beta$ be
the Hoffman map with $L(0)=\xi$ by Lemma \ref{lem2.2},
where $(L_\beta)$ is a convergent subnet of $\{L_j\}$. Then
$\mathbf{H}(\Delta)$ is the algebra of all $h\circ L(\zeta)$
for $h$ in $B(\mathfrak{F}^\sharp)$. So \eqref{eq5.6} holds for
$H(\zeta)= h\circ L(\zeta)$, by regarding $H(\zeta)$ as a function
on $\Delta$. Although $\xi$ and
$\phi$ are usually different, the values $h(\phi)$ and $h(\xi)$
coincide for all $h$ in $B(\mathfrak{F}^\sharp)$.
Moreover we see that the measure $\widehat{\nu}$
in \eqref{eq5.3} may be regarded as a representing
measure $\widehat{m}$ on homeomorphic part
$P(\xi)$, and the sequence of $u^\sharp_j(\theta) dm(\theta)$
may be replaced by the one of Poisson kernels $P_{c_j}$
on $\mathbf{T}$.
Since $B(\phi) = 0 = G(0)$, $H \to h(\phi)$ is the evaluation
homomorphism at $0$ for $A(\Delta)$, which extends uniquely to the one on $H^\infty(\Delta)$, so does for $\mathbf{H}(\Delta)$. Thus it seems to be
natural that \eqref{eq5.6} holds.
However we do not know whether $\mathbf{H}(\Delta)$ is uniformly closed,
although it is strictly smaller than $H^\infty(\Delta)$.
This fact follows easily from Lusin's theorem, since
$\mathbf{H}(\Delta)$ is the algebra generated by
$A(\Delta)$ and finite set $\{F_k(\zeta)\}$
with $F_k(\zeta)=\lim_{j\to \infty} f_k\circ L_j(\zeta)$.

\bigskip
\bigskip
\section{Remarks}
\label{S3}

(a) \; As far as we restrict our attention to analytic discs
in $\mathfrak{M}(H^\infty)\setminus \Delta$, it is rather
easy to show that each of them belongs to the closure of
$\Delta$ in $\mathfrak{M}(H^\infty)$.  Indeed, let $P(\phi)$
be a nontrivial Gleason part in $\mathfrak{M}_1$, and let $\mu$
be the representing measure for $\phi$ on the Shilov boundary
$X$. Denote by $H^p(\mu), 1\le p <\infty,$ the closure of
$H^\infty(\Delta)$ in $L^p(\mu)$. Then Wermer's embedding theorem
assures the existence of an inner function $Z$ in $H^2(\mu)$ such
that $Z$ has a bijective extension $\widehat{Z}$ to $P(\phi)$ with
$\widehat{Z}(\phi)=0$, with which $\tau(z)= \widehat{Z}^{-1}(z)$
is an analytic map on $\Delta$, meaning that $f\circ\tau(z)$ is
analytic on $\Delta$ for all $f$ in $H^\infty(\Delta)$, and
\begin{equation}
\label{eq6.1}
f(\xi)\;=\;\sum_{n=0}^{\infty}\, a_n \widehat{Z}^n(\xi)\;,
\quad \xi \in P(\phi),
\end{equation}
(see, for example, \cite[Chapter 6, \S 6.4]{L}). Since $Z$ is an
inner function in $H^2(\mu)$, there is a sequence $\{q_i\}$ in
$H^\infty(\Delta)$ such that $\vert q_i(x) \vert \leq 1$ on $X$
and $\Vert q_i - Z \Vert_{L^2(\mu)} \to 0,$ as $i\to\infty$
(see the proof of \cite[Chapter II, Theorem 7.2]{G1}).
Let $f_1,f_2, \cdots, f_N $ be in $H^\infty(\Delta)$, and put
$\mathfrak{F}\,=\, \{f_1,f_2, \cdots, f_N \}\,\cup\,\{\, q_i
\,;\, i = 1,2, \cdots \}$. Let $\varepsilon >0$, and denote by
$f_k\circ \tau(z)=\sum_{n=0}^{\infty}\, a_n^{(k)}\,z^n$ the Taylor
expansion of $f_k\circ \tau$ on $\Delta$. By \eqref{eq6.1} we
choose a $q_i$ in $\{q_i\}$ and an integer $\ell_k \ge 0$ such that
\begin{equation*}
\left \vert \phi(f_k)\;-\;\sum_{n=0}^{\ell_k}\, a_n^{(k)}
q_i^n(\phi) \right \vert \;<\; \varepsilon, \quad k = 1,2, \cdots, N.
\end{equation*}
It follows from Lemmas \ref{lem3.2} and \ref{lem3.3} that we find a
thin interpolating sequence $\{c_j\}$ for which the maps $L_j$ by
\eqref{eq2.2} satisfy that there is a sequence $\{u_j(\theta)\}$ of
nonnegative functions such that $\int u_j(\theta)\, dm(\theta) =1$ and
\begin{equation*}
\left\vert \phi(f) - \int_{-\pi}^\pi f\circ L_j(\theta)u_j(\theta)
\,dm(\theta)\right \vert \;\to \; 0\,, \quad j\to \infty,
\end{equation*}
for all $f$ in $\mathfrak{F}$. Taking a subsequence of $\{c_j\}$
suitably, we may assume that each $f$ in $\mathfrak{F}$ satisfies
$f\circ L_j(\zeta)$ converges uniformly on compact subsets in $\Delta$.
Let $L$ be the Hoffman map by a convergent subnet of $\{L_j\}$ in
$\mathfrak{M}(H^\infty)^\Delta$. Then $P_1=L(\Delta)$ is an
analytic disc homeomorphic to $\Delta$. We then see that
$e^{i\gamma_j} q_j\circ L(\zeta)$, with some real $\gamma_j$,
converges to $G(\zeta)=\zeta$, so the measure
$u_j(\theta)\,dm(\theta)$ converges to the representing measure
$dm(\theta)$ at $\zeta=0$. Then the sequence $\{c_j\}$ in $\Delta$
has the property that $f(c_j)\to f(\phi)$ for $f$ in $\mathfrak{F}$,
consequently, we see that
$$
W(\phi, f_1,\cdots, f_N, \varepsilon)\cap \Delta\,\neq\,\emptyset,
$$
for any $\varepsilon >0$.

By a similar argument as above, it enables us to show Hoffman's
characterization of analytic discs by interpolating sequences.
Under the notation above, let $\mathfrak{F}_m = \{q_1, q_2, \cdots, q_m\}$.
Then there is a thin interpolating sequence $\{\zeta^{(m)}_j\}$ in
$\Delta$ and a homeomorphic part $P(\phi^{(m)})$ such that
$\lim_{j\to\infty} g(\zeta^{(m)}_j) = g(\phi^{(m)})=g(\phi)$ for any $g$
in the algebra generated by $\mathfrak{F}_m$. It suffices to consider
the case where $\{\zeta^{(m)}_j\}\cap \{\zeta^{(k)}_j\}=\emptyset$
if $m\neq k$. With the aide of diagonal argument, we choose an
interpolating sequence $\{\zeta_j\}$ so that
$$
\{\zeta_j\}\;=\;\{\zeta^{(m)}_{j_m}, \zeta^{(m)}_{j_m+1},
\cdots ,\zeta^{(m)}_{j_{m+1}-1} \;;\;
m=1, 2, \cdots \},
$$
for which such $\phi^{(m)}$ may be regarded as an adherent
point for $\{\zeta_j\}$ in $\mathfrak{M}(H^\infty)$.
Modifying the argument in Lemma \ref{lem3.3}, we may see
that $P(\phi)$ lies in the closure
of $\cup_{m=1}^\infty P(\phi^{(m)})$ in $\mathfrak{M}(H^\infty)$.
Thus $\{\zeta_j\}$ is a desired interpolating sequence.
Conversely, if $\phi$ lies in the closure of an interpolating
sequence $\{\zeta_j\}$, then we see easily that $\phi$ lies
in a nontrivial Gleason part $P(\phi)$. Then by a similar
argument as above, we find a Hoffman map $L$ by converging
subnet of $\{L_j\}$ such that $P(\phi)=L(\Delta)$. For our
characterization of analytic discs, full details and further
developments will appear elsewhere.

\medskip
(b) \; We may represent concretely a large portion of the fiber
$\mathfrak{M}_1$ by a continuous flow. It is useful to study
representing measures in connection with invariant measures.
For analyticity in ergodic theory, we refer to
\cite{M1}, \cite{M2} and \cite{P}. Let $H^\infty(\mathbf{R}^2_+)$
be the space of all bounded analytic functions on the upper
half-plane ${\bf R}^2_+$. Then
$H^\infty({\bf R}^2_+)$ is identified with $H^\infty(\Delta)$
via the conformal map $z(w) =(w-i)/(w+i)$. Setting
$g(w)\,=\,f(z(w))$ for $f$ in $H^\infty(\Delta)$, we should
investigate the behavior of $g(w)$ around at infinity to look
into the structure of $\mathfrak{M}_1$. Let $\beta \mathbf{Z}$
be the Stone-\v{C}ech compactification of the integer group
$\mathbf{Z}$. Then the shift operator $S_0\,n\,=\,n+1$ on
$\mathbf{Z}$ extends to a homeomorphism $S$ on $\beta\mathbf{Z}$.
Let $\mathbf{X}$ be the quotient space obtained from
$\beta {\bf Z}\times [0,1]$ by identifying $(y,1)$ with $(Sy,0)$.
By regarding the real line $\mathbf{R}$ as $\mathbf{Z}\times [0,1)$
in $\mathbf{X}$, the translation on $\mathbf{R}$ induces a
continuous flow $(\mathbf{X},\{S_t\}_{t\in\mathbf{R}})$ by
\begin{equation*}
S_t(y,s)\;=\;(S^{[s+t]}y, s+t-[s+t]),\;\quad (y,s)\in \mathbf{X},
\end{equation*}
where $[t]$ denotes the largest integer not exceeding $t$. We write
$\mathbf{x}$ for $(y,s)$ in $\mathbf{X}$, and the translate
$S_t\,\mathbf{x}$ is denoted by $\mathbf{x}+t$. Let $A(\mathbf{X})$
be the uniform algebra of all functions $f$ in $C(\mathbf{X})$
satisfying that each $t\to f(\mathbf{x}+t)$ lies in
$H^\infty(\mathbf{R})$, the space of all boundary value functions
in $H^\infty(\mathbf{R}^2_+)$. Then $A(\mathbf{X})$ is a logmodular
algebra on $\mathbf{X}$ whose maximal ideal space is identified with
a certain quotient space of $\mathbf{X}\times [0,\infty]$. Recall
that the Poisson kernel $P_{ir}$ for ${\bf R}^2_+$ is defined by $P_{ir}(t)=r/\pi(t^2+r^2)$.
For a bounded Borel function $g$ on $\mathbf{X}$, we put
\begin{equation*}
g(\mathbf{x},r)\,=\, g\ast P_{ir}(\mathbf{x})\,=\,
\int_{-\infty}^\infty g(\mathbf{x}+t) P_{ir}(t)dt,\quad
(\mathbf{x},r) \in \mathbf{X}\times (0,\infty).
\end{equation*}
This decides the representing measures for $A(\mathbf{X})$ on
$\mathbf{X}\times (0,\infty)$, while representing measures
lying in $\mathbf{X}\times \{\infty\}$ are invariant
measures being multiplicative on $A(\mathbf{X})$ (see \cite{M2},
\cite{T1}, \cite{T2} for representing measures for $A(\mathbf{X})$ ).
Denote by $H^\infty(\mathbf{X})$ the algebra of all bounded Borel
functions $g$ for which $\mathbf{x}\to g(\mathbf{x},r)$ lies in
$A(\mathbf{X})$ for each $r>0$. Since $H^\infty(\mathbf{X})$ is
isometrically isomorphic to $H^\infty(\Delta)$, the subset
$\mathbf{X}\times (0,\infty]\setminus \mathbf{R}^2_+$ represents
a portion of the fiber $\mathfrak{M}_1$, from which we observe
immediately that either \textit{nontangential} point or
\textit{orocycular} point is in the closure of an interpolating
sequence (compare with \cite[Chapter X, Excersises 1 and 2]{Ga}).
Let $\mathbf{M}$ be a minimal set in
$(\mathbf{X},\{S_t\}_{t\in\mathbf{R}})$ (see \cite{P} for minimal
sets). Observe that each $\textit{O}(\mathbf{x})\times (0,\infty)$
corresponds to an analytic disc, where $\textit{O}(\mathbf{x})$
denotes the orbit $\{\mathbf{x}+t ;\, t \in \mathbf{R}\}$. If
$\mathbf{x}$ is in $\mathbf{M}$, then the analytic disc by
$\textit{O}(\mathbf{x})\times (0,\infty)$ is never homeomorphic to $\Delta$.
We also see that every representing measure on $\mathbf{M}$ not being
point mass has the same support set $\mathbf{M}$, on which there are
many representing measures. Since $\mathbf{M}$ is an intersection
of peak sets, the restriction $A_\mathbf{M}$ of $A(\mathbf{X})$ to
$\mathbf{M}$ is a uniform algebra equipped with many interesting
properties (see \cite{M1}, \cite{M2} and \cite{T3} for more details).

\end{document}